\documentclass[11pt]{amsart}
\usepackage{epsfig}
\usepackage{amsbsy,latexsym}
\usepackage{amsmath}
\usepackage{amssymb, mathrsfs}
\usepackage[mathscr]{eucal}
\usepackage{hyperref}
\usepackage{amsfonts}
\usepackage{amsthm}
\usepackage{amsmath}
\usepackage{amsfonts}
\usepackage{latexsym}
\usepackage{amssymb}
\usepackage{amscd}
\usepackage[latin1]{inputenc}
\usepackage{verbatim}
\usepackage[english]{babel}
\usepackage{xcolor}

\newtheorem{definition}{Definition}[section]
\newtheorem{theorem}[definition]{Theorem}
\newtheorem{lemma}[definition]{Lemma}
\newtheorem{corollary}[definition]{Corollary}
\newtheorem{proposition}[definition]{Proposition}
\theoremstyle{definition}

\newtheorem{example}[definition]{Example}

\newcommand{\C}{\mathbb{C}}

\newcommand\tr{ \operatorname{Tr} }

\title[Entanglement Breaking Rank]{Entanglement Breaking Rank via Complementary Channels and Multiplicative Domains}

\begin{document}

\author[D.~W. Kribs, J.~Levick, R. Pereira, M.~Rahaman]{David~W.~Kribs$^{1,2}$, Jeremy Levick$^{1,2}$, Rajesh Pereira$^{1}$, Mizanur Rahaman$^{3}$}

\address{$^1$Department of Mathematics \& Statistics, University of Guelph, Guelph, ON, Canada N1G 2W1}
\address{$^2$Institute for Quantum Computing, University of Waterloo, Waterloo, ON, Canada N2L 3G1}
\address{$^3$ Univ Lyon, ENS Lyon, UCBL, CNRS, Inria, LIP, F-69342, Lyon Cedex 07, France}




\begin{abstract}
Quantum entanglement can be studied through the theory of completely positive maps in a number of ways, including by making use of the Choi-Jamilkowski isomorphism, which identifies separable states with entanglement breaking quantum channels, and optimal ensemble length with entanglement breaking rank. The multiplicative domain is an important operator structure in the theory of completely positive maps. We introduce a new technique to determine if a channel is entanglement breaking and to evaluate entanglement breaking rank, based on an analysis of multiplicative domains determined by complementary quantum channels. We give a full description of the class of entanglement breaking channels that have a projection as their Choi matrix, and we show the entanglement breaking and Choi ranks of such channels are equal.
\end{abstract}

\maketitle


\section{Introduction}

Quantum entanglement remains one of the most important and deeply studied topics in modern quantum information theory. It is a broad topic that can be studied through a number of different perspectives, including via the theory of completely positive maps and the Choi-Jamilkowski isomorphism, which identifies separable bipartite quantum states with the special class of quantum channels (completely positive, trace-preserving maps) called entanglement breaking \cite{holevo1998coding,horodecki2003entanglement}. The identification also links the important measure of optimal ensemble cardinality of a separable state with the entanglement breaking rank of its associated channel.

Thus, the notion of entanglement breaking rank has received recent and growing interest. Another reason for this, is that when considering some special classes of entanglement breaking channels, the entanglement breaking rank has a surprising connection with the geometry of the rank-one Kraus operators that define the channels. For example, as was shown in \cite{pandey2020entanglement},  the entanglement breaking rank of a particular quantum channel on $M_d(\mathbb{C})$ is $d^2$ if and only if there exist a set of $d^2$ equiangular lines in $\C^d$. Note that the latter existence statement is still an open problem in general and a popular conjecture of Zauner (see references in \cite{pandey2020entanglement}).

In this paper, we first introduce a new technique to determine if a channel is entanglement breaking based on an analysis of complementary quantum channels \cite{holevo2012quantum,horodecki2003entanglement,holevo2007complementary}. In the theory of completely positive maps, the multiplicative domain of a map \cite{choi1974schwarz,paulsen2002completely} is an important operator structure that has arisen in a number of relatively recent investigations in quantum information theory \cite{choi2009multiplicative,johnston2011generalized,kribs2018quantum}. We make use of multiplicative domains associated with complementary quantum channels to evaluate the entanglement breaking rank of entanglement breaking channels that have Choi matrix equal to a projection, showing that the entanglement breaking and Choi ranks of the channel are equal.

This paper is organized as follows. The rest of this section includes a presentation of the preliminary notions we require at the intersection of operator theory and quantum information. In Section~2 we recall the formulation of complementary quantum channels, and we present a description of when maps are entanglement breaking in terms of the complementary channel perspective. In Section~3, we begin by briefly recalling the multiplicative domain from the theory of completely positive maps, and then we derive our main result, equating the entangling breaking rank to the Choi rank for entanglement breaking channels with Choi matrix a projection, and we finish the section by giving a full description of what channels satisfy the result. We conclude in Section~4 with some remarks on related work and possible extensions of the approach.

\subsection{Preliminaries}

Let $M_n(\C)$ be the set of $n\times n$ complex matrices for a fixed positive integer $n \geq 1$. A linear map $\Phi : M_n(\C) \rightarrow M_m(\C)$ is called a (quantum) {\it channel} if it is completely positive and trace-preserving \cite{holevo2012quantum,nielsen}. By a {\it state} or {\it density matrix} we mean a positive semi-definite matrix with trace equal to one.

In order to be consistent with the interpretation of dual vectors as row vectors, we take inner products to be linear in the second element, so that if $v = \begin{pmatrix} v_1 \\ \vdots \\ v_n\end{pmatrix}$, we have $v^* = (\overline{v_1}, \ \cdots, \ \overline{v_n})$ and $v^*w = \langle v,w\rangle$. Hence, $uv^*$ will denote the rank-one operator, whose action is defined in a way consistent with the notation; e.g., $(uv^*)w = \langle v,w\rangle u$. 

Every channel $\Phi : M_n(\C) \rightarrow M_m(\C)$ has an operator-sum decomposition (non-unique) given by,
\begin{align}
\Phi(X) & = \sum_{i=1}^d K_i X K_i^* \quad \quad \forall X \in M_n(\C),  \label{krausform}
\end{align}
with the {\it Kraus operators} $K_i\in M_{m\times n}(\C)$ satisfying the trace-preserving condition $\sum_i K_i^* K_i = I_n$.

The minimal number $d$ of Kraus operators in an operator-sum form Equation~(\ref{krausform}) for $\Phi$ is called the {\it Choi rank} of $\Phi$, which we will denote by $\operatorname{Choi-rank}(\Phi)$. This rank is so-called, as it is equal to the rank of the {\it Choi matrix} $J(\Phi)\in M_n(\C) \otimes M_m(\C)$, where
\[
J(\Phi) = \sum_{i,j=1}^n E_{ij} \otimes \Phi(E_{ij}),
\]
and with the `matrix units' for $M_n(\C)$ given by $E_{ij} = e_ie_j^*$ for a fixed orthonormal basis $\{e_1,\cdots, e_n\}$ for $\C^n$.

When viewed as a map, $J(\cdot)$ defines the {\it Choi-Jamiolkowski isomorphism} between completely positive maps from $M_n(\C)$ to $M_m(\C)$ and positive semi-definite operators in $M_n(\C) \otimes M_m(\C)$. For a channel, the Choi matrix is a density matrix when scaled by the factor $n^{-1}$. The isomorphism can be used to study quantum entanglement; in particular because the channels that have a (normalised) separable Choi matrix are precisely the channels that break entanglement in the following sense.

\begin{definition}
A channel $\Phi : M_n(\C) \rightarrow M_m(\C)$ is {\it entanglement breaking} if for all $k\geq 1$ and states $X\in M_k(\C)\otimes M_n(\C)$, the output state $(\mathrm{id}_k \otimes \Phi)(X)$ is separable, where $\mathrm{id}_k$ is the identity map on $M_k(\C)$.
\end{definition}

This class of channels has been heavily investigated in recent years following the original works  \cite{holevo1998coding,horodecki2003entanglement} (for example, see recent related work of the authors and collaborators \cite{rahaman2018eventually,rahaman2019new,girard2020mixed,kribs2021nullspaces,ahiable2021entanglement}).  There are a number of equivalent characterisations of entanglement breaking channels beyond the definition; most importantly in our present setting is the rank-one form,
\begin{align}
\Phi(X)
&= \sum_{i=1}^r (u_iv_i^*)X(u_iv_i^*)^*\label{rankone},
\end{align}
for some set of vectors $u_i\in\C^m$, which without loss of generality we take to be unit vectors, and a set of vectors $v_i\in\C^n$ that satisfy $\sum_i v_i v_i^* = I_n$ due to trace-preservation.

Our investigation focusses on the following notion for such maps \cite{pandey2020entanglement}.

\begin{definition}
Given an entanglement-breaking channel $\Phi$, we call
the minimum $r$ such that Equation~(\ref{rankone}) holds the \emph{entanglement breaking rank} of $\Phi$, or $\operatorname{EB-rank}(\Phi)$ for short.
\end{definition}

For an entanglement breaking channel $\Phi$, it is not hard to see that the number $\operatorname{EB-rank}(\Phi)$ is equal to the {\it length} or {\it optimal ensemble cardinality} of the corresponding separable state $J(\Phi)$ (i.e., the minimal number of pure states required to write the state as a convex linear combination of those pure states).

It is clear that the Choi rank is bounded above by the EB-rank, since it is equal to the minimal number of Kraus operators required in an operator-sum form for the map.
It is quite hard to compute the EB-rank of an entanglement breaking channel as one has to search through all the rank-one Kraus decompositions of the map and find one with the least number of them. In the rest of the paper we identify a class of maps for which these ranks coincide.


\section{Complementary Channels and a Characterisation of Entanglement Breaking Channels}

We next discuss the complement of a quantum channel \cite{holevo2012quantum,horodecki2003entanglement,holevo2007complementary} and show how its behaviour can be used to detect when a map is entanglement breaking. We finish by focussing on the situation considered in the next section.

\begin{definition} Given a channel $\Phi:M_n(\C) \rightarrow M_m(\C)$ with Kraus operators $\{K_i\}_{i=1}^d$, the \emph{complementary channel}, or \emph{complement}, is the map $\Phi^C : M_n(\C) \rightarrow M_d(\C)$ defined  by
\begin{equation}\label{complement}
\Phi^C(X):=\sum_{i,j=1}^d \mathrm{tr}(K_i^*K_jX)E_{ji}.
\end{equation}
\end{definition}

There are a number of easily derived properties of complementary channels, including the following which we recall now as it will be used in what follows. We remind the reader of the (Hilbert-Schmidt) dual map $\Phi^\dagger$ for a map $\Phi$, defined via the trace relation; $\tr(\Phi^\dagger(X) Y) := \tr(X \Phi(Y))$.

There is never a unique set of Kraus operators for a channel, but (as described in \cite{nielsen}) if $\{K_i\}_{i=1}^d$ is a minimal set (and so $d$ is necessarily the Choi rank of $\Phi$), and $\{L_j\}_{j=1}^r$ are any other set of Kraus operators, then there exists an isometry $W:\C^d \rightarrow \C^r$, such that if $w_i$ are the columns of $W^*$ with coordinates $w_i = (w_{ji})$, we have 
\begin{equation}\label{changeofKraus} L_i = \sum_{j=1}^d \overline{w_{ji}}K_j.\end{equation}

Note that changing Kraus operators corresponds to adjunction of the complement by the isometry $W$:
\begin{align*} \sum_{i,j=1}^r \mathrm{tr}(L_j^*L_iX)E_{ij} &= \sum_{i,j=1}^r \sum_{k,l=1}^d \overline{w_{ki}}w_{jl} \mathrm{tr}(K_l^*K_kX)E_{ij} \\
& = \sum_{j=1}^r \overline{w_{ki}}E_{ik}\biggl(\sum_{k,l=1}^d \mathrm{tr}(K_l^*K_kX)E_{kl}\biggr) \sum_{i=1}^r w_{jq}E_{lj}\\
& = \sum_{i,p=1}^{r,d} \overline{w_{pi}}E_{ip}\biggl(\sum_{k,l=1}^d \mathrm{tr}(K_l^*K_kX)E_{kl}\biggr)\sum_{q,j=1}^{d,r} w_{jq}E_{qj}\\
& = W\biggl(\sum_{k,l=1}^d \mathrm{tr}(K_l^*K_kX)E_{kl}\biggr)W^*.
\end{align*}
For the rest of the paper, when we talk of `the' complementary channel for a given channel $\Phi$, we will be referring to the complement defined by a minimal set of Kraus operators for the map. By the above argument, such maps are equivalent up to adjunction by a unitary map.

As a useful observation in the calculations to come, note that Equation~(\ref{complement}) immediately implies that
\[
\mathrm{tr}(\Phi^{C\dagger}(E_{ij})X)  = \Phi^C(X)_{ji}  = \mathrm{tr}(K_i^*K_jX),
\]
and as this holds for all $X$, we have for all $1 \leq i,j \leq d$,
\begin{equation}\label{compadjoint}
\Phi^{C\dagger}(E_{ij}) = K_i^*K_j.
\end{equation}
The following result makes use of this fact to give a characterization of entanglement breaking channels in terms of complementary channels.

\begin{proposition}\label{compadjointrankone}
Let $\Phi : M_n(\C) \rightarrow M_m(\C)$ be a channel of Choi rank $d$. Then $\Phi$ is entanglement breaking if and only if there exist vectors $\{w_i\}_{i=1}^r\subseteq \C^d$ such that
\begin{equation}\label{resofidentity}\sum_{i=1}^r w_iw_i^* = I_d ,
\end{equation}
and there exist vectors $\{v_i\}_{i=1}^r \subseteq \C^n$ such that
\begin{equation} \Phi^{C\dagger}(w_iw_i^*) = v_iv_i^* \ \forall i.\end{equation}
\end{proposition}

\begin{proof}
Suppose $\Phi$ is entanglement breaking; then as per Equation (\ref{rankone}) we may find a set of Kraus operators for $\Phi$ of the form $\{u_iv_i^*\}_{i=1}^r$. If $\{K_i\}_{i=1}^d$ are a minimal set of Kraus operators for $\Phi$, then by Equation (\ref{changeofKraus}) there must exist an isometry $W$, such that with $w_i$ the columns of $W^*$,
\begin{equation}\label{changetorankone} u_iv_i^*  = \sum_{j=1}^d \overline{w_{ji}}K_j.\end{equation}

Then multiply both sides of Equation (\ref{changetorankone}) by the respective adjoints (and recalling our default assumption that each $u_i$ is a unit vector) to get
\begin{align*} v_i v_i^* = \|u_i\|^2 v_iv_i^* &= \sum_{j,k=1}^d \overline{w_{ji}}w_{ki}K_k^*K_j \\
& = \sum_{j,k=1}^d w_{ki} \overline{w_{ji}} \Phi^{C\dagger}(E_{kj}) \\
& = \Phi^{C\dagger}\biggl(\bigl(\sum_{k=1}^d w_{ki}e_k\bigr)\bigl(\sum_{j=1}^d \overline{w_{ji}}e_j^*\bigr)\biggr) \\
& = \Phi^{C\dagger}(w_iw_i^*),
\end{align*}
and where we use Equation (\ref{compadjoint}) to replace $K_i^*K_j$ with $\Phi^{C\dagger}(E_{ij})$.
Since $W$ is an isometry,
$$W^*W = \sum_{i=1}^r w_iw_i^* = I_d.$$

For the converse, note that $\Phi^{C\dagger}(ww^*)$ for any (column) vector $w\in\C^d$ is $K(\overline{w})^*K(\overline{w})$ where $K: \C^d \rightarrow M_{m\times n}(\C)$ is defined by
\[
K(x) : = \sum_{i=1}^d x_iK_i ,
\]
for any $x\in \C^d$ with coordinates $x_i$.

Thus, $\Phi^{C\dagger}(ww^*) = vv^*$ for some vector $v$  if and only if $K(\overline{w})\in M_{m\times n}(\C)$ is a rank-one matrix. If there exist $\{w_iw_i^*\}$ each with $\Phi^{C\dagger}(w_iw_i^*) = v_iv_i^*$, then we may form the matrix $W^*$ whose columns are $w_i$; $W^*W = \sum_{i=1}^r w_iw_i^*$. So $W$ is an isometry if and only if Equation (\ref{resofidentity}) holds. If it does, then $\{L_i = K(\overline{w_i})\}_{i=1}^r$ is also a set of Kraus operators for $\Phi$, and so together we have a set of rank-one Kraus operators for $\Phi$. These must be of the form $L_i = u_iv_i^*$ for some vectors $u_i,v_i$, completing the proof.
\end{proof}

We can use this result to frame entanglement breaking rank as follows.

\begin{corollary}\label{EBrankcor}
Let $\Phi: M_n(\C)\rightarrow M_m(\C)$ be an entanglement breaking channel with Choi rank $d$. Then $\operatorname{EB-rank}(\Phi)$  is the smallest $r$ such that there exist $r$ vectors $\{w_i\}_{i=1}^r \subseteq \C^d$ and $\{v_i\}_{i=1}^r \subseteq \C^n$   satisfying
\[
\sum_{i=1}^r w_iw_i^* = I_d; \quad \Phi^{C\dagger}(w_iw_i^*) = v_iv_i^* \quad  \forall 1 \leq  i \leq r.
\]
\end{corollary}

Before proceeding to our main result in the next section, we link specific behaviour of the complement adjoint with structure of the Choi matrix.

\begin{lemma}\label{projectiontp}
Let $\Phi: M_n(\C) \rightarrow M_m(\C)$ be a channel with $\operatorname{Choi-rank}(\Phi) = d$ and complementary channel $\Phi^C : M_n(\C) \rightarrow M_d(\C)$. The complement adjoint $\Phi^{C\dagger}$ is trace-preserving if and only if the Choi matrix $J(\Phi)$ is  a  projection.
\end{lemma}

\begin{proof}
Let $\{K_i\}_{i=1}^d$ be a set of Kraus operators in $M_{m \times n}(\C)$ for $\Phi$, and further assume $\tr(\Phi^{C\dagger}(X)) =  \tr(X)$ for all $X\in M_d(\C)$. As this holds for all $X$, we have $\Phi^C(I_n) =  I_d$ and so,
\[
\Phi^C(I_n)  = \sum_{i,j=1}^d \tr (K_j^*K_i)E_{ij}  =  I_d .
\]
Hence, it follows that the operators $K_i$ are orthogonal in the trace inner product; specifically, $\tr(K_j^* K_i) = \delta_{ij}$.

The Kraus operators are related to the Choi matrix by defining vectors $k_i = (I_n\otimes K_i)\phi\in \C^n \otimes \C^m$ where $\phi = \sum_{i=1}^n e_i \otimes e_i$, so that $k_i = \mathrm{vec}(K_i)$, which is a block vector whose blocks are the columns of $K_i$ in order. Then the Choi matrix is given by $J(\Phi) = \sum_{i=1}^d k_ik_i^*$.

It is clear that $\tr (K_i^*K_j) = \langle k_i,k_j\rangle$, and so from the above calculation it follows that the vectors $\{  k_i \}_{i=1}^d$ form an orthonormal set. Hence we have $J(\Phi) = \sum_{i=1}^d k_ik_i^* = P$, where $P$ is a rank-$d$ projection.

For the converse direction, the above argument can be reversed to show that $J(\Phi)$ being a projection implies $\Phi^C$ is unital and hence that its adjoint is trace-preserving.
\end{proof}

\section{The Multiplicative Domain of a Completely Positive Map and Entanglement Breaking Rank}

We now discuss the multiplicative domain of a unital completely positive map, and make use of its structure to identify entanglement breaking ranks.

\begin{definition}
Let $\Phi: M_n(\C) \rightarrow M_m(\C)$ be a completely positive map. The \emph{multiplicative domain} of $\Phi$, $\mathcal{M}(\Phi)$, is the set

\begin{equation} \label{MD} \mathcal{M}(\Phi):=\{A: \Phi(AX) = \Phi(A)\Phi(X) \ \& \ \Phi(XA) = \Phi(X)\Phi(A) \ \forall X\}.\end{equation}

The left- and right-multiplicative domains, $\mathcal{M}_L(\Phi)$, $\mathcal{M}_R(\Phi)$, are respectively defined by

\begin{align} \mathcal{M}_L(\Phi) &= \{A: \Phi(AX) = \Phi(A)\Phi(X) \ \forall X\}, \\
 \mathcal{M}_R(\Phi) &= \{A: \Phi(XA) = \Phi(X)\Phi(A) \ \forall X\}.
\end{align}
\end{definition}

The key internal description of the multiplicative domain, as derived originally by Choi, is given as follows. This is Theorem $3.1$ of \cite{choi1974schwarz}, and we include a proof that is slightly different from the original.

\begin{theorem} If $\Phi$ is a unital completely positive map, then $A \in \mathcal{M}_L(\Phi)$ if and only if
\begin{equation}\label{AAstar} \Phi(AA^*) = \Phi(A)\Phi(A^*)\end{equation} and $A \in \mathcal{M}_R(\Phi)$ if and only if
\begin{equation} \label{AstarA} \Phi(A^*A) = \Phi(A^*)\Phi(A).\end{equation}
\end{theorem}

\begin{proof}  One direction is trivial, so we show the other; we only show the claim for the left-domain, as the proof of the other claim is basically identical. Let $V = \begin{bmatrix} I & A^* & X^*\end{bmatrix}$ so that $V^*V = \begin{bmatrix} I & A^* & X^* \\ A & AA^* & AX^* \\ X & XA^* & XX^*\end{bmatrix}$ is positive semidefinite.

Since $\Phi$ is completely positive, it is in particular $3$-positive, so applying $\Phi$ to the blocks must yield a positive semidefinite matrix as well; as $\Phi$ is unital we get that $$\begin{bmatrix}I & \Phi(A^*) & \Phi(X^*) \\ \Phi(A) & \Phi(AA^*) & \Phi(AX^*) \\ \Phi(X) & \Phi(XA^*) & \Phi(XX^*)\end{bmatrix}$$ is positive semidefinite. Take the Schur complement with respect to the top-left block to see that

$$\begin{bmatrix} \Phi(AA^*) -\Phi(A)\Phi(A^*) & \Phi(AX^*) - \Phi(A)\Phi(X^*) \\ \Phi(XA^*) - \Phi(X)\Phi(A^*) & \Phi(XX^*) - \Phi(X)\Phi(X^*)\end{bmatrix}$$ is positive semidefinite as well.

But if $\Phi(AA^*) = \Phi(A)\Phi(A^*)$, the top-left block of this positive semidefinite matrix is $0$, so the blocks below and right of it must be zero as well. Hence we have 
$\Phi(AX^*) = \Phi(A)\Phi(X^*)$ for all $X$.
\end{proof}

The case of projections is of particular importance to us.

\begin{corollary}\label{projectioncor} A projection $P$ is in the multiplicative domain of a  unital completely positive map $\Phi$ if and only if $\Phi(P)$ is also a projection.
\end{corollary}

This follows as $\Phi(P^*P) = \Phi(PP^*) = \Phi(P) = \Phi(P)\Phi(P)^*=\Phi(P)^*\Phi(P)$, so that Equations (\ref{AAstar}, \ref{AstarA}) are both satisfied simultaneously.

Thus, it follows that the multiplicative domain of a unital completely positive map $\Phi$ is a unital $*$-subalgebra (a finite-dimensional C$^*$-algebra), and hence is unitarily equivalent to $\bigoplus_{k=1}^M I_{i_k}\otimes M_{j_k}(\C)$ for some (unique) set of pairs of positive integers $(i_k,j_k)$ \cite{choi1974schwarz}. 

We now apply our results from the previous section in this multiplicative domain context.

\begin{lemma}\label{mdlemma}
Suppose $\Phi: M_n(\C) \rightarrow M_m(\C)$ is an entanglement breaking channel whose Choi matrix is a projection of rank $d$. Let $\{w_i\}_{i=1}^r\subseteq \C^d$ be the vectors guaranteed by Proposition \ref{compadjointrankone} that satisfy
\[
\sum_{i=1}^r w_iw_i^* = I_d ,
\]
and for all $1 \leq i \leq r$,
\[
\Phi^{C\dagger}(w_iw_i^*) = v_iv_i^*.
\]
Then $w_iw_i^* \in \mathcal{M} (\Phi^{C\dagger})$ for all $i$.
\end{lemma}

\begin{proof}
By Lemma \ref{projectiontp}, the map $\Phi^{C\dagger}$ is trace-preserving. As $\Phi$ is a channel, so is $\Phi^C$ and hence $\Phi^{C\dagger}$ is also unital. It follows that,
\[
\|w_i\|^2  =  \tr (w_iw_i^*) 
 = \tr (\Phi^{C\dagger}(w_iw_i^*))  = \tr (v_iv_i^*) 
= \|v_i\|^2,
\]
and hence $||w_i|| = ||v_i||$.
This means that $w_iw_i^*$ and $v_iv_i^*$ can be simultaneously re-scaled so they are both projections on their respective Hilbert spaces. That is,
$P_i = \frac{w_iw_i^*}{\|w_i\|^2}$ and $Q_i = \frac{v_i v_i^*}{\|w_i\|^2} = \frac{v_iv_i^*}{\|v_i\|^2}$ are both projections, and they satisfy
\[
\Phi^{C\dagger}(P_i) = Q_i,
\]
for each $i$. Thus by Corollary \ref{projectioncor}, $w_iw_i^*$ are all in the multiplicative domain $\mathcal{M} (\Phi^{C\dagger})$ of $\Phi^{C\dagger}$, as claimed.
\end{proof}

We can now prove our main result.





\begin{theorem}\label{mainthm}
Let $\Phi$ be a channel whose Choi matrix $J(\Phi)$ is a rank-$d$ projection. Then $\Phi$ is entanglement breaking if and only if $\mathcal{M} (\Phi^{C\dagger})$ is unitarily equivalent to an algebra of the form $\bigoplus_{k=1}^M M_{j_k}(\C)$. Moreover, in this case the entanglement breaking rank is the same as the Choi rank;
\[
\operatorname{EB-rank}(\Phi) = \operatorname{Choi-rank}(\Phi) = d.
\]
\end{theorem}

\begin{proof}
Suppose first that the multiplicative domain $\mathcal{M}(\Phi^{C\dagger})$ is unitarily equivalent to $\oplus_{k=1}^M M_{j_k}(\C)$. In particular then, this algebra contains an algebra unitarily equivalent to the diagonal algebra, and hence contains rank-one projections unitarily equivalent to the diagonal matrix units, $E_{ii} = e_ie_i^*$. Thus, up to unitary equivalence, $e_ie_i^*$ are in the multiplicative domain of $\Phi^{C\dagger}$, which note is unital following from the fact that $\Phi$ (and $\Phi^C$) is trace-preserving. Hence by Corollary~\ref{projectioncor} we have that $P_i:=\Phi^{C\dagger}(e_ie_i^*)$ is also a projection. 

Moreover, from Lemma~\ref{projectiontp}, we have that $J(\Phi)$ being a projection is equivalent to $\Phi^{C\dagger}$ being trace-preserving, as well as unital. Thus the rank of $P_i$ is $\mathrm{tr}(P_i) = \mathrm{tr}(\Phi^{C\dagger}(e_ie_i^*)) = 1$, and so in fact $P_i$ is a rank-one projection for each $i$; that is, there is some unit vector $v_i$ such that $\Phi^{C\dagger}(e_ie_i^*) = v_iv_i^*$. 
It is also obvious that $\sum_i e_ie_i^* = I$, and so we satisfy the conditions of Proposition~\ref{compadjointrankone}, and thus $\Phi$ is entanglement breaking. 

For the converse direction, assume that $\Phi$ is entanglement breaking. We also have that $J(\Phi)$ is a projection, and so again by Lemma~\ref{projectiontp}, it follows that $\Phi^{C\dagger}$ is both trace-preserving and unital. From Proposition~\ref{compadjointrankone}, we know that $\Phi$ being entanglement breaking implies the existence of vectors $\{w_i\}, \{v_i\}$ satisfying $\sum_i w_iw_i^* = I$ and  $\Phi^{C\dagger}(w_iw_i^*) = v_iv_i^*$. Further, from Lemma~\ref{mdlemma} we have that the vectors $\{w_iw_i^*\}$ are in the multiplicative domain of $\Phi^{C\dagger}$. 
Thus, $\mathcal{M}(\Phi^{C\dagger})$ contains rank-one projections summing to the identity. 

From the structure theory for matrix algebras, we know that the multiplicative domain of $\Phi^{C\dagger}$ must have the form, up to unitary equivalence, $\oplus_{k=1}^M I_{i_k}\otimes M_{j_k}(\C)$; we now show that if such an algebra contains a set of rank-one projections summing to the identity, that $i_k = 1$ for all $k$.

Suppose $P_i$ are a set of projections in such an algebra summing to the identity; so it must be that each $P_i = \oplus_{k=1}^M I_{i_k}\otimes P_{ik}$ where each $P_{ik}$ is a projection in $M_{j_k}(\C)$, and furthermore we must have that $\sum_i P_{ik} = I_{j_k}$ for each $k$. Moreover, the rank of $P_i$ is of course just its trace:
\[
\mathrm{tr}(P_i) = \sum_{k=1}^m i_k \mathrm{tr}(P_{ik}) = \sum_k i_k \mathrm{rank}(P_{ik}).
\] 
Thus, if $P_i$ are rank-one projections summing to the identity, it must be that $1 = \sum_k i_k \mathrm{rank}(P_{ik})$ and this is clearly only possible if $\mathrm{rank}(P_{ik}) = 0$ for all but one $k$, in which case $\mathrm{rank}(P_{ik}) = i_k = 1$. Hence for each $i$ there is an index $1\leq q_i \leq M$ such that  
$P_i = \oplus_{k=1}^m \delta_{k,q_i} P_{k q_i}$, and $i_{q_i} = 1$. So, we will be done if we can show that the map $i \mapsto q_i$ is a surjection onto $\{1,2,\cdots, M\}$. Suppose not; then there is a direct summand, the $q^{th}$ one say, on which none of the $P_i$ are supported; $P_i$ restricted to the $q^{th}$ direct summand is always 0 for all $i$. But then, $\sum_i P_i$ restricted to that direct summand will also be 0, not $I_{k_q}$, as required for the projections to sum to the identity. So in fact, for each $q$, there is an $i$ such that $q = q_i$, and thus $i_{q_i} = 1$. Hence the algebra has the form $\oplus_{k=1}^M I_1 \otimes M_{j_k} = \oplus_{k=1}^M M_{j_k}(\C)$ as claimed.

Finally, we show that the EB-rank is the same as the Choi rank of $\Phi$, which is $d$. We now know that for a channel with $J(\Phi)$ equal to a projection, being entanglement breaking is equivalent to $\Phi^{C\dagger}$ having a multiplicative domain (up to unitary equivalence) of the form $\oplus_{k=1}^M M_{j_k}(\C)$. Moreover, we know that the matrix unit projections $E_{ii} = e_i e_i^*$ inside the algebra form a set of rank-one projections for which $\Phi^{C\dagger}(e_ie_i^*) = v_iv_i^*$, and so by Corollary~\ref{EBrankcor}, it is thus clear that the EB rank is no greater than $d$. For it to be smaller than $d$, we would have to find a set of fewer than $d$ vectors such that $\sum_i w_iw_i^* = I_d$, but this is clearly impossible, so in fact the EB rank must be exactly $d$. 
\end{proof}

A simple class of examples that Theorem \ref{mainthm} applies to arise from consideration of {\it Schur product channels}. These are channels $\Phi:M_d(\C)\rightarrow M_d(\C)$ of the form
\[
\Phi(X) = X\circ C , 
\]
where $\circ$ indicates the element-wise Schur product, and $C$ is a correlation matrix; a positive semidefinite matrix all of whose diagonal entries are equal to one. Such channels arise in quantum information \cite{levick2017quantum,harris2018schur}; for instance, any random unitary channel implemented by mutually commuting unitaries is also a Schur product channel in some basis.

As an observation we will use below, let us note how the double complement of a map may be identified with the map itself. Indeed, suppose we have a minimal set of Kraus operators for a map $\Phi$, given by $\{K_i\}_{i=1}^n$; so we know that any other minimal set of Kraus operators gives a complement related to the first by $U\Phi^C(X)U^*$ for some unitary $U$. Then suppose $\{L_i\}_{i=1}^{n'}$ are a minimal set of Kraus operators for $\Phi^C$; picking any other set of minimal Kraus operators for $\Phi$ will give instead the channel $U\Phi^C(X)U^*$ which clearly has Kraus operators $\{UL_i\}_{i=1}^{n'}$. 
Now, using the formula for the complement, for any $U$, we can write the complement of the complement: 
\[
(U\Phi^C(X)U^*)^C = \sum_{i,j=1}^{n'} \mathrm{tr}((UL_i)^*(UL_j)X) E_{ji} = \sum_{i,j} \mathrm{tr}(L_i^*L_j X) E_{ji}; 
\]
that is, the presence of the unitary adjunction by $U$ does not change the complement; in fact, even if $U$ is an isometry rather than a unitary, this remains true, so that our choice of Kraus operators for $\Phi$ does not affect our formula for $(\Phi^C)^C$ at all; only a difference choice of Kraus operators for $\Phi^C$ will change this. 
Thus, we can define $(\Phi^C)^C(X)$ up to adjunction to be $V\sum_{i,j} \mathrm{tr}(L_i^*L_jX)E_{ji} V^*$ for any unitary $V$, and some fixed minimal set of Kraus operators for $\Phi^C$. 

Returning to Schur product channels, note that any such channel is automatically unital, and so its complement has the property that $J(\Phi^C)$ is a projection by Lemma~\ref{projectiontp} (making use of the observation above). Further, it is not too hard to show that if the correlation matrix $C$ is the Gram matrix for unit vectors $\{\overline{v_i}\}$, that is, $c_{ij} = \langle \overline{v_i},\overline{v_j}\rangle$, then the complement has the form
\[
\Phi^C(X) = \sum_{i=1}^d x_{ii} v_iv_i^* = \sum_{i=1}^d v_ie_i^* X e_iv_i^* ,
\]
and so is clearly entanglement-breaking, with rank-one Kraus operators $v_ie_i^*$ (and $e_i$ the standard basis vectors for $\mathbb{C}^d$). In this case, $J(\Phi^C)$ is just the projection $P = \bigoplus_{i=1}^d v_iv_i^*$, and by Theorem~\ref{mainthm} the EB-rank of $\Phi^C$ is equal to $d$.

In fact, we can show that such channels are the only ones that satisfy Theorem \ref{mainthm}. By `external' and `internal' unitary equivalence to a map $\Phi$, we mean, respectively, there exist unitaries $U$, $V$ and some map $\Phi^\prime$ such that $\Phi^\prime (X) = U \Phi(X) U^*$ and $\Phi^\prime(X) = \Phi (V X V^*)$ for all $X$. 

\begin{theorem}\label{mainthmconvs} An entanglement breaking channel $\Phi$ has its Choi matrix equal to a projection if and only if $\Phi$ is internally or externally unitarily equivalent to the complement of a Schur product channel.
\end{theorem}

\begin{proof} The only thing left to establish the forward direction is to prove the claim above that the complement of a Schur product channel has the desired properties. If $\Phi(X) = X\circ C$ where $c_{ij} = \langle v_i,v_j\rangle$ for some unit vectors $\{v_i\}_{i=1}^n \in \C^d$, then $\Phi$ has $d$ Kraus operators $K_i = \sum_{j=1}^n \overline{v_{ji}}E_{jj}$;
\begin{align*} \Phi(X) & = \sum_{i=1}^d \sum_{j,k} \overline{v_{ji}}v_{ki} E_{jj}XE_{kk}\\
& = \sum_{j,k=1}^n x_{jk} E_{jk}\bigl(\sum_{i=1}^d \overline{v_{ji}}v_{ki}\bigr)\\
& = \sum_{j,k=1}^n x_{jk}E_{jk}\langle v_j,v_k\rangle \\
& = \sum_{j,k=1}^n c_{jk}x_{jk}E_{jk}.
\end{align*}
Then $\Phi^C$ can be written as
\begin{align*}\Phi^C(X) & = \sum_{i,j=1}^d \mathrm{tr}(K_i^*K_jX)E_{ji} \\
&=\sum_{i,j,k,l=1}^{d,n}v_{ki}\overline{v_{lj}} \mathrm{tr}(E_{kk}E_{ll}X)E_{ji}\\
&=\sum_{i,j,k=1}^{d,n} v_{ki}\overline{v_{kj}} x_{kk} E_{ji}\\
&=\sum_{k=1}^n x_{kk} \bigl(\sum_{j,i=1}^d v_{ki}\overline{v_{kj}}E_{ji} \big) \\
&=\sum_{k=1}^n x_{kk} \overline{v_k}v_k^T.
\end{align*}
This clearly has Kraus operators $\{\overline{v_k}e_k^*\}_{k=1}^n$ and so is obviously an entanglement breaking channel. Further, the Choi matrix of this is then $J(\Phi^C) = \sum_{i=1}^n E_{ii}\otimes v_iv_i^*$ and since each $v_i$ is a unit vector, this is a projection. Notice also that the Choi rank of $J(\Phi^C)$ is $n$, as is the EB-rank, as witnessed by the $n$ rank-one Kraus operators $\{\overline{v_k}e_k^*\}$.

Now we show the converse. By Theorem \ref{mainthm}, $\Phi$ must have EB-rank equal to its Choi rank, which is also necessarily the dimension of the underlying space. This last fact follows from the fact that the rank of $J(\Phi)$, the Choi rank, is also the trace of $J(\Phi)$; but by trace-preservation we must have that $(\mathrm{id}\otimes \mathrm{tr})\bigl(J(\Phi)\bigr) = I_d$ and hence
\[\mathrm{tr}(J(\Phi)) = (\mathrm{tr}\otimes \mathrm{tr})\bigl(J(\Phi)\bigr) = \mathrm{tr}(I_d).\]

Thus, there exist $d$ unit vectors $\{u_i\}$ and $d$ vectors $\{v_i\}$ such that
\begin{equation} \Phi(X) = \sum_{i=1}^d u_iv_i^* X v_i u_i^*.\end{equation}
Imposing trace-preservation on the Kraus operators requires
\[
\sum_{i=1}^d v_iu_i^*u_i v_i^* = \sum_{i=1}^d v_iv_i^* = I_d,
\]
and the only way to have $d$ rank-ones adding up to the identity on $\mathbb{C}^d$  is for each to be a projection onto a subspace orthogonal to all the others; hence $\{v_i\}$ is an orthonormal basis for $\C^d$. Thus, if we form the channel $\Phi_V(X) = \Phi(VXV^*)$, where $V$ is the unitary matrix defined by $Ve_i = v_i$, we see that \[\Phi_V(X) = \sum_i u_i v_i^*VXV^*v_i u_i^* = \sum_{i=1}^d u_i e_i^* X e_iu_i^*.\]

Now, we calculate the complement of $\Phi_V^C$:
\begin{eqnarray*}
\Phi_V^C(X) &=& \sum_{i,j=1}^d \mathrm{tr}(e_iu_i^*u_je_j^*X)E_{ji} \\
&=& \sum_{i,j=1}^d \langle u_i,u_j\rangle \langle e_j, X e_i\rangle E_{ji}
= \sum_{i,j=1}^d \langle u_i,u_j\rangle x_{ji}E_{ji} ,
\end{eqnarray*}
and it is clear that this is just $X$ with the $(j,i)$ entry multiplied by $\langle u_i,u_j\rangle =\overline{c_{ij}}$ where $C$ is the Gram matrix for the unit vectors $\{u_i\}$. Thus, $C$ is positive semidefinite with $1$'s down the diagonal, and so indeed we have that \[\Phi_V^C(X) = \overline{C}\circ X\] for the correlation matrix $\overline{C}$, and the claim is proved.
\end{proof}

\section{Concluding Remarks}

As noted above, computing entanglement breaking ranks is important in the theory of quantum entanglement in general, but quite challenging for specific classes of channels, or equivalently for separable states. Thus, even though Theorem~\ref{mainthm} applies to a limited class of channels, as shown by Theorem~\ref{mainthmconvs}, it still contributes to the short list of instances for which this rank can be computed explicitly.

The use of complementary channel properties in the study of entanglement breaking ranks is novel and could conceivably lead to further progress for other classes. Applications from an analysis of the structure of the multiplicative domain is also new in this setting, and is reminiscent of other applications in quantum information noted above of such operator structures. Further study of potential uses of multiplicative domain structures in entanglement theory are warranted.

All of this said, the possible extensions of the specific  approach presented here to wider classes of entanglement breaking channels would appear to be limited.
For instance, we have attempted to extend our results to channels with Choi matrix equal to a scalar multiple of a projection. Some of our results do extend to this more general setting. As an example, we can extend Lemma~\ref{projectiontp} to show the following, where we say that a map $\Phi$ is {\it trace-stabilizing} if for some fixed non-zero scalar $\alpha$, we have $\tr(\Phi(X)) = \alpha \tr(X)$ for all $X$ in the domain of $\Phi$.

\begin{lemma}\label{projection}
The complement adjoint $\Phi^{C\dagger}$ of a channel $\Phi$ is trace-stabilizing if and only if the Choi matrix $J(\Phi)$ is a scalar multiple of a  projection.
\end{lemma}

An important example of a channel that satisfies this condition, is the {\it completely depolarizing channel} \cite{nielsen} $\Phi_{\mathrm{CD}} : M_n(\C)\rightarrow M_n(\C)$ defined by $\Phi_{\mathrm{CD}}(X) = \frac{\tr(X)}{n} I_n$. Note this channel has Choi matrix equal to $\frac1n I_{n^2}$, and has Choi rank and EB rank equal to $n^2$.

However, Theorem \ref{mainthm} does not generalize to the case of a channel with Choi matrix equal to a scalar multiple of a projection. A counterexample can be given as follows:
Consider the Werner-Holevo channel $\Phi: M_d(\C)\rightarrow M_d(\C)$ defined by
$$\Phi(X)=\frac{\tr(X)+X^t}{d+1},$$
where $X^t$ is the transpose of X. This is known to be an entanglement breaking map, and its Choi matrix $J(\Phi)$ is the bipartite operator $\frac{I \otimes I +W}{d+1}$, where $W=\sum_{i,j}E_{i,j}\otimes E_{j,i}$. Now note that $P=(I \otimes I+W)/2$ is a projection in $M_d(\C)\otimes M_d(\C)$ and hence
$$J(\Phi)=\frac{2}{d+1}P,$$
which means it is multiple of a projection and it follows that its Choi rank is $\frac{d(d+1)}{2}$. But this channel has EB rank equal to $d^2$ as shown in \cite{pandey2020entanglement}. In fact, separating the EB rank and the Choi rank is a non-trivial task and this is one of only a few examples are known where these two ranks are different.

\vspace{0.1in}

{\noindent}{\it Acknowledgements.} We are grateful to the referee for helpful comments.  D.W.K. was supported by NSERC Discovery Grant 400160. R.P. was supported by NSERC Discovery Grant 400550. M.R. is supported by the European Research Council (ERC Grant Agreement No. 851716). Part of this work was completed during IWOTA Lancaster UK 2021, and as such, the authors would like to express thanks to the conference organisers and acknowledge their funding via EPSRC grant EP/T007524/1.

\bibliographystyle{plain}
\bibliography{KLPRbib}

\begin{thebibliography}{10}

\bibitem{ahiable2021entanglement}
Jennifer Ahiable, David~W Kribs, Jeremy Levick, Rajesh Pereira, and Mizanur
  Rahaman.
\newblock Entanglement breaking channels, stochastic matrices, and primitivity.
\newblock {\em Linear Algebra and its Applications}, 629:219--231, 2021.

\bibitem{choi1974schwarz}
Man-Duen Choi.
\newblock A {S}chwarz inequality for positive linear maps on {C}*-algebras.
\newblock {\em Illinois Journal of Mathematics}, 18(4):565--574, 1974.

\bibitem{choi2009multiplicative}
Man-Duen Choi, Nathaniel Johnston, and David~W Kribs.
\newblock The multiplicative domain in quantum error correction.
\newblock {\em Journal of Physics A: Mathematical and Theoretical},
  42(24):245303, 2009.

\bibitem{girard2020mixed}
Mark Girard, Debbie Leung, Jeremy Levick, Chi-Kwong Li, Vern Paulsen, Yiu~Tung
  Poon, and John Watrous.
\newblock On the mixed-unitary rank of quantum channels.
\newblock {\em arXiv preprint arXiv:2003.14405}, 2020.

\bibitem{harris2018schur}
Samuel~J Harris, Rupert~H Levene, Vern~I Paulsen, Sarah Plosker, and Mizanur
  Rahaman.
\newblock Schur multipliers and mixed unitary maps.
\newblock {\em Journal of Mathematical Physics}, 59(11):112201, 2018.

\bibitem{holevo1998coding}
Alexander~S Holevo.
\newblock Coding theorems for quantum channels.
\newblock {\em Russian Mathematical Surveys}, 53:1295--1331, 1999.

\bibitem{holevo2007complementary}
Alexander~S Holevo.
\newblock Complementary channels and the additivity problem.
\newblock {\em Theory of Probability \& Its Applications}, 51(1):92--100, 2007.

\bibitem{holevo2012quantum}
Alexander~S Holevo.
\newblock {\em Quantum Systems, Channels, Information: A Mathematical
  Introduction}, volume~16.
\newblock Walter de Gruyter, 2012.

\bibitem{horodecki2003entanglement}
Michael Horodecki, Peter~W Shor, and Mary~Beth Ruskai.
\newblock Entanglement breaking channels.
\newblock {\em Reviews in Mathematical Physics}, 15(06):629--641, 2003.

\bibitem{johnston2011generalized}
Nathaniel Johnston and David~W Kribs.
\newblock Generalized multiplicative domains and quantum error correction.
\newblock {\em Proceedings of the American Mathematical Society},
  139(2):627--639, 2011.

\bibitem{kribs2018quantum}
David~W Kribs, Jeremy Levick, Mike~I Nelson, Rajesh Pereira, and Mizanur
  Rahaman.
\newblock Quantum complementarity and operator structures.
\newblock {\em Quantum Information \& Computation}, 19:67--83, 2019.

\bibitem{kribs2021nullspaces}
David~W Kribs, Jeremy Levick, Katrina Olfert, Rajesh Pereira, and Mizanur
  Rahaman.
\newblock Nullspaces of entanglement breaking channels and applications.
\newblock {\em Journal of Physics A: Mathematical and Theoretical},
  54(10):105303, 2021.

\bibitem{levick2017quantum}
Jeremy Levick, David~W Kribs, and Rajesh Pereira.
\newblock Quantum privacy and {S}chur product channels.
\newblock {\em Reports on Mathematical Physics}, 80(3):333--347, 2017.

\bibitem{nielsen}
Michael~A. Nielson and Isaac~L. Chuang.
\newblock {\em Quantum Computation and Quantum Information}.
\newblock Cambridge University Press, 2000.

\bibitem{pandey2020entanglement}
Satish~K Pandey, Vern~I Paulsen, Jitendra Prakash, and Mizanur Rahaman.
\newblock Entanglement breaking rank and the existence of {SIC} {POVM}s.
\newblock {\em Journal of Mathematical Physics}, 61(4):042203, 2020.

\bibitem{paulsen2002completely}
Vern Paulsen.
\newblock {\em Completely bounded maps and operator algebras}.
\newblock Number~78. Cambridge University Press, 2002.

\bibitem{rahaman2019new}
Mizanur Rahaman.
\newblock A new bound on quantum {W}ielandt inequality.
\newblock {\em IEEE Transactions on Information Theory}, 66(1):147--154, 2019.

\bibitem{rahaman2018eventually}
Mizanur Rahaman, Samuel Jaques, and Vern~I Paulsen.
\newblock Eventually entanglement breaking maps.
\newblock {\em Journal of Mathematical Physics}, 59(6):062201, 2018.

\end{thebibliography}


\end{document}